\documentclass[12pt]{article}%
\usepackage[intlimits]{amsmath}
\usepackage{amssymb}
\usepackage{latexsym}
\usepackage{tikz}
\usepackage{anysize}
\usepackage{subfigure}
\usepackage[T1]{fontenc}
\usepackage[sc]{mathpazo}
\usepackage{color}
\usepackage[colorlinks]{hyperref}
\usepackage{amsfonts}
\usepackage{graphicx}%
\definecolor {refcol}{RGB}{40,0,255}
\hypersetup{colorlinks=true,allcolors=refcol}
\makeatletter
\newfont{\footsc}{cmcsc10 at 8truept}
\newfont{\footbf}{cmbx10 at 8truept}
\newfont{\footrm}{cmr10 at 10truept}
\newtheorem{theorem}{Theorem}

\newtheorem{conjecture}{Conjecture}
\newtheorem{corollary}{Corollary}

\newtheorem{example}{Example}

\newtheorem{lemma}{Lemma}

\newtheorem{proposition}[theorem]{Proposition}

\newtheorem{remark}{Remark}

\newenvironment{proof}[1][Proof]{\noindent{\textbf {#1}}}  {\hfill$\Box$\bigskip}

\begin{document}
\title{\textbf{New bounds on the distance Laplacian and distance signless Laplacian spectral radii}}
\author{Roberto C. D\'{i}az, Ana I. Julio and Oscar Rojo \thanks{Department of
Mathematics, Universidad Cat\'{o}lica del Norte, Antofagasta, Chile}}

\date{}
\maketitle

\maketitle
\begin{abstract}
Let $G$ be a simple undirected connected graph. In this paper, new upper bounds on the distance Laplacian spectral radius of $G$ are obtained. Moreover, new lower and upper bounds for the distance signless Laplacian spectral radius of $G$ are derived. Some of the above mentioned bounds are sharp and the graphs attaining the corresponding bound are characterized. Several illustrative examples are included.
\end{abstract}


\section{Introduction}


Let $G=\left( V(G),E(G) \right) $ be a simple undirected graph on $n$ vertices with vertex set $V(G)=\{v_1,\ldots,v_n\}$ and edge set $E(G)$.


The distance between $u,v \in V(G)$ for a connected graph $G$, denoted by $d(u,v)$, is the length of the shortest path connecting $u$ and $v$. The Wiener index $W(G)$ of a connected graph $G$ is
\begin{equation*}
  W(G) = \frac{1}{2} \sum_{u,v \in V(G)}d(u,v)
\end{equation*}
and the transmission $Tr(v)$ of a vertex $v \in V(G)$ is the sum of the distances from $v$ to all other vertices of $G$, that is,
\begin{equation*}
  Tr(v)= \sum_{u \in V(G)} d(v,u).
\end{equation*}

The graph $G$ is said to be $k-$ \emph{transmission regular} if $Tr(v)=k$ for each vertex $v \in V(G)$.

The \emph{distance matrix} of a graph $G$ of order $n$ is the $n \times n$ matrix $\mathcal{D}(G) = \left(d_{i,j}\right)$, indexed by the vertices of $G$,
where $d_{i,j}=d(v_{i},v_{j})$. Two of the oldest works on the distance matrix are \cite{schonberg1935} (1935) and \cite{young1938} (1938). In \cite{hakimi_yau_64} (1964), the authors ask for the conditions under which a given a square symmetric matrix
with real and nonnegative entries is realized as the distance matrix of a graph. Different realizability problems of distance matrices are investigated in \cite{spereira90,spereira87,spereira69,spereira66}, \cite{Buneman74}, \cite{varone98} and \cite{boesh68-69}, among many others papers. A remarkable result on the distance matrix, due to Graham and Pollak \cite{graham_pollak71} (1971), is : If $T$ is a tree of order $n \ge 2$ with distance matrix $\mathcal{D}(T)$, then
\begin{equation*}
\text{det} \mathcal{D}(T) = (-1)^{n-1}\left(n-1\right)2^{n-2}. \label{graham_pollak}
\end{equation*}
Thus, the determinant of the distance matrix of a tree depends only on its order. The result established by Graham and Pollak attracted much interest among algebraic graph theory researchers.

The eigenvalues of $\mathcal{D}(G)$ are called the \emph{distance eigenvalues} of $G$ and denoted by
\begin{equation*}
\partial_1(G) \geq \partial_2(G)\geq \ldots \geq \partial_n(G).
\end{equation*}
 Some properties about the eigenvalues of the distance matrix are already known. For example, in \cite{Zho}, the author characterizes the graphs with minimal spectral radius of the distance matrix in three classes of simple connected graphs with $n$ vertices: with fixed vertex connectivity, matching number and chromatic number, respectively. In \cite{lin}, the authors characterize all connected graphs with \ $\partial_{n}(G)=-2$. Furthermore, they characterize all connected graphs of diameter $2$ with exactly three distance eigenvalues when \ $\partial_{1}(G)$ is not an integer. They also conjecture that the complete $k$-partite graph is determined by its distance spectrum. Later, in \cite{Yu}, the author determines all graphs which satisfy \ $\partial_{n}(G)\in [-2.383 \ , \ 0]$. A very complete survey of the state of the art on distance matrices up to 2014 appears in \cite{AuchichHansen2014} (see also
\cite{StevanovicIlic2010}). \\

 In \cite{Aou} Aouchiche and Hansen introduce, for a connected graph $G$, the distance Laplacian matrix $\mathcal{L}(G)$ and the distance signless Laplacian matrix $\mathcal{Q}(G)$ as follows
\begin{equation*}
  \mathcal{L}(G) = \emph{Tr}(G)-\mathcal{D}(G)
\end{equation*}
and
\begin{equation*}
  \mathcal{Q}(G) = \emph{Tr}(G)+\mathcal{D}(G)
\end{equation*}
where
\begin{equation*}
  \emph{Tr}(G)=\emph{diag}[\emph{Tr}(v_1),\emph{Tr}(v_2),\ldots,\emph{Tr}(v_n)]
\end{equation*}
 is the diagonal matrix of the vertex transmissions in $G$.

In \cite{Aou}, among other results, the above mentioned authors prove the equivalence between the distance signless Laplacian, distance Laplacian and the distance spectra for the class of transmission regular graphs. The eigenvalues of $\mathcal{L}(G)$ and $\mathcal{Q}(G)$ are called the \emph{distance Laplacian eigenvalues} and the \emph{distance signless Laplacian eigenvalues} of $G$ and they are denoted by
\begin{equation*}
  \partial_1^{L}(G) \geq \partial_2^{L}(G)\geq \ldots \geq \partial_n^{L}(G)
\end{equation*}
and
\begin{equation*}
  \partial_1^{Q}(G) \geq \partial_2^{Q}(G)\geq \ldots \geq \partial_n^{Q}(G),
\end{equation*}
respectively.

$\mathcal{L}(G)$ and $\mathcal{Q}(G)$ are both real symmetric matrices. From this fact and
Ger\v{s}gorin's Theorem, it follows that their eigenvalues are nonnegative real numbers. Let $\textbf{1}$ be the all ones vector. Clearly each row sum of $\mathcal{L}(G)$ is $0$. Then $\left(0,\mathbf{1}\right) $ is an eigenpair of $\mathcal{L}(G)$ and, since $G$ is connected graph, $0$ is a simple eigenvalue.

Moreover
\begin{equation*}
  trace(\mathcal{L}(G)) = trace(\mathcal{Q}(G)) = 2 W(G).
\end{equation*}

Clearly, if $G$ is a $k-$ transmission regular graph then
\begin{equation*}
  \mathcal{L}(G) = k I_n - \mathcal{D}(G)
\end{equation*}
and
\begin{equation*}
  \mathcal{Q}(G) =k I_n + \mathcal{D}(G)
\end{equation*}
where $I_n$ is the identity matrix of order $n$; and, for $i=1,\ldots,n$,
\begin{equation*}
  \partial_i^{L}(G) = k-\partial_{n-i+1}(G)
\end{equation*}
and
\begin{equation*}
  \partial_i^{Q}(G) = k + \partial_i(G)
\end{equation*}

From the Perron-Frobenius Theory for nonnegative matrices, we have
\begin{theorem}
\label{PF}
If $A$ is a nonnegative matrix then its spectral radius $\rho(A)$ is an eigenvalue of $A$, and it has an associated nonnegative eigenvector. Furthermore, if $A$ is irreducible then $\rho(A)$ is a simple
eigenvalue of $A$ with an associated positive eigenvector.
\end{theorem}

In particular, since $\mathcal{Q}(G)$ is a positive matrix, $\rho(\mathcal{Q}(G))$ is a simple eigenvalue of $\mathcal{Q}(G)$.

The Frobenius norm of an $n \times n$ matrix $M=(m_{i,j})$ is
\begin{equation*}
  \|M\|_{F} = \sqrt{\sum_{i=1}^{n}\sum_{j=1}^{n} \vert m_{i,j} \vert ^{2}}.
\end{equation*}
We recall if $M$ is a normal matrix then
$\|M\|_{F}^{2} = \sum_{i=1}^{n} \vert \lambda_i(M)\vert^{2}$ where $\lambda_1(M),\ldots,\lambda_n(M)$ are the eigenvalues of $M$. In particular,
\begin{equation*}
  \|\mathcal{L}(G)\|_{F}^{2}=\sum_{i=1}^{n-1}(\partial_i^{L}(G))^{2}
\end{equation*}
and
\begin{equation*}
  \|\mathcal{Q}(G)\|_{F}^{2}=\sum_{i=1}^{n}(\partial_i^{Q}(G))^{2}.
\end{equation*}



Throughout this paper, we assume that $G$ is a connected graph of order $n$ and $K_n$ denotes the complete graph on n vertices. \\

Some lower and upper bounds on $\partial_1^{L}(G)$ and $\partial_1^{Q}(G)$ are already known. The purpose of this paper is to search for new bounds on these spectral radii. Section 2, is dedicated to upper bounds on $\partial_1^{L}(G)$, we recall some known upper bounds as well as some previous results that allow us to derive new upper bounds. Finally, in Section 3, we recall some known results and we obtain new lower and upper bounds on $\partial_1^{Q}(G)$. In Section 2 as well as in Section 3, illustrative examples are given.

\section{Bounds on $\partial_1^{L}(G)$}

A basic result on $\partial_1^{L}(G)$ is given in Corollary 3.6 of \cite{Aou}:
\begin{theorem}
\label{t21}
Let $G$ be a connected graph of order $n\geq 3$. Then
\begin{equation*}
  \partial_i^{L}(G) \geq \partial_i^{L}(K_n)=n
\end{equation*}
for $i=1,2,\ldots,n-1$, and
\begin{equation*}
  \partial_n^{L}(G)=\partial_n^{L}(K_n)=0.
\end{equation*}
\end{theorem}

A basic result, due to Aouchiche and Hansen, concerning a graph with only two distinct distance Laplacian eigenvalues is

\begin{theorem}\label{Han}
(\cite{Hansen}, Theorem 2.7) If $G$ is a graph on $n > 2$ vertices then the multiplicity of $\partial_1^{L}(G)$ is less or equal to $n-1$ with equality if and only if $G=K_n$.
\end{theorem}

Some upper bounds on $\partial_1^{L}(G)$ are already known. Among them, we have
\begin{theorem}
(\cite{You}, Theorem 3.7) Let $G=(V,E)$ be a connected graph on n vertices. Then
\begin{align}\label{i1}
\partial^{L}_{1}(G) \ \leq \ \max\limits_{1\leq i \leq n}\bigg\{Tr(v_{i})+\sqrt{(n-1)\sum\limits_{k=1}^{n}d^{2}_{k,i}}\bigg\} \ .
\end{align}
Moreover, if the equality in (\ref{i1}) holds, then \ $Tr(v_{i})+\sqrt{(n-1)\sum\limits_{k=1}^{n}d^{2}_{k,i}} \ = \ Tr(v_{j})+\sqrt{(n-1)\sum\limits_{k=1}^{n}d^{2}_{k,j}}$ \ for any $i,j\in \{1,\ldots,n\}$.
\end{theorem}
\begin{theorem}
(\cite{Das}, Theorem 4.3) Let $G$ be a connected graph on $n\geq 4$ vertices. Then
\begin{equation}\label{d1}
  \partial_1^{L}(G) \leq 2 W(G)-n(n-2).
\end{equation}
\end{theorem}

\begin{theorem} (\cite{Das}, Theorem 4.4)
Let $G$ be a connected graph on $n$ vertices. Then
\begin{equation}\label{d2}
  \partial_1^{L}(G) < \max_{1\leq i \leq n}\emph{Tr}(v_i) + \sqrt{\|\mathcal{D}(G)\|_{F}^{2}
  - \frac{\sum_{i=1}^{n}\big(\emph{Tr}(v_i)\big)^{2}}{n}}.
\end{equation}
\end{theorem}

In the following propositions, we derive some new upper bounds on $\partial_1^{L}(G)$. We begin recalling a useful result due to Brauer \cite{brauer} :

\begin{theorem}
\label{BRAUER}Let $A$ be an $n\times n$ arbitrary matrix with eigenvalues
\begin{equation*}
\lambda _{1},\ldots,\lambda_k,\dots,\lambda _{n}.
\end{equation*}
Let
\begin{equation*}
\mathbf{x}=\left[ x_{1},\ldots,x_{n}\right] ^{T}
\end{equation*}%
be an eigenvector of $A$ corresponding to the eigenvalue $\lambda _{k}$ and
let $\mathbf{q}$ be any $n$-dimensional column vector. Then the matrix $A+%
\mathbf{xq}^{T}$ has eigenvalues
\begin{equation*}
\lambda _{1},\lambda _{2},\ldots,\lambda _{k-1},\lambda _{k}+\mathbf{x}^{T}%
\mathbf{q},\lambda _{k+1},\ldots,\lambda _{n}.
\end{equation*}
\end{theorem}

We next derive our first new upper bound on $\partial_1^{L}(G)$.

\begin{theorem}
\label{cota1}
Let $G$ be a connected graph of order $n$. For $i=1,\ldots,n$, let
\begin{equation*}
p_{i}=\max_{j\neq i}d_{i,j}
\end{equation*}
and let
\begin{equation}\label{B}
\mathcal{B}(G) = \mathcal{L}(G)+\mathbf{1}\mathbf{p}^{T}.
\end{equation}%
Then
\begin{equation}\label{n1}
\partial _{1}^{L}(G) \leq \sum_{i=1}^{n}p_i.
\end{equation}
If the equality in (\ref{n1}) holds then $\mathcal{B}(G)$ is a reducible matrix. This necessary condition for the equality in (\ref{n1}), it is not a sufficient condition. If $\mathcal{B}(G)$ is a irreducible matrix the inequality in (\ref{n1}) is strict.
\end{theorem}

\begin{proof} \ \ Since $(0,\mathbf{1}) $ is an eigenpair
for the distance Laplacian matrix $\mathcal{L}(G),$ using Theorem \ref%
{BRAUER}, we obtain that the eigenvalues of $\mathcal{B}(G)$
are
\begin{equation*}
\partial_1^{L}(G),\ldots,\partial_{n-1}^{L}(G), \mathbf{1}^{T}\mathbf{p}.
\end{equation*}%
Moreover
\begin{equation*}
\mathcal{B}(G) \mathbf{1}= \mathcal{L}(G) \mathbf{1}+\mathbf{1}(\mathbf{p}^{T}\mathbf{1})=(\mathbf{p}^{T}\mathbf{1})\mathbf{1}= (\mathbf{1}^{T}\mathbf{p})\mathbf{1}.
\end{equation*}
The entries of $\mathcal{B}(G)=(b_{i,j})$ are%
\begin{equation*}
b_{i,i} = Tr(v_i) + p_i
\end{equation*}
for $i=1,\ldots,n$
and
\begin{equation*}
  b_{i,j}=-d_{i,j}+p_j
\end{equation*}
for $j \neq i$.
Then $\mathcal{B}(G)$ is a nonnegative matrix.
From Theorem \ref{PF}, we obtain that $\rho (\mathcal{B}(G))=
\mathbf{1}^{T}\mathbf{p}=\sum_{i=1}^{n}p_i$ with eigenvector the all
ones vector $\mathbf{1}.$
Therefore $\partial_1^{L}(G) \leq \sum_{i=1}^{n}p_i.$

Suppose that $\partial_1^{L}(G) = \sum_{i=1}^{n}p_i.$ Hence $\rho (\mathcal{B}(G))$ is a repeated eigenvalue of the nonnegative matrix $\mathcal{B}(G)$ and then, from Theorem \ref{PF}, it is a reducible matrix. The converse is not true (see Example \ref{ex1} below).

Suppose that $\mathcal{B}(G)$ is a irreducible matrix. From Theorem \ref{PF}, $\rho(\mathcal{B}(G))$ is a simple eigenvalue and then $\partial_1^{L}(G)$ <  $\rho(\mathcal{B}(G))=\sum_{i=1}^{n}p_i$.
\end{proof}

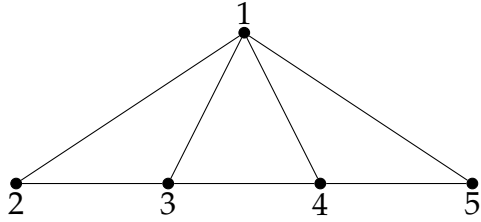
\begin{example}
\label{ex1}
Let $G:$

\begin{center}
\begin{tikzpicture}
 \tikzstyle{every node}=[draw,circle,fill=black,minimum size=4pt,
                            inner sep=0pt]
    \draw (0,2) node (1) [label=above:$1$] {}
          (-3,0) node (2) [label=below:$2$] {}
         (-1,0) node (3) [label=below:$3$] {}
          (1,0) node (4) [label=below:$4$] {}
         (3,0) node (5) [label=below:$5$] {};
        \draw (1)--(2);  \draw (1)--(3); \draw (1)--(4);  \draw (1)--(5);        \draw (2)--(3); \draw (3)--(4);         \draw (4)--(5);
\end{tikzpicture}
\end{center}

To four decimal places, the distance Laplacian eigenvalues of $G$ are $0$, $5$, $5.5858$, $7$ and $8.4142$. We have $\textbf{p}=[1,2,2,2,2]^{T}$ and then $\partial_1^{L}(G) < \sum_{i=1}^{5}p_i=9$. Since $p_1=1$, the first column of $\mathcal{B}(G)$ is $[5,0,0,0,0]^{T}$ which shows that $\mathcal{B}(G)$ is reducible.
\end{example}

There are also some results on upper bounds for the second largest modulus $\xi \left( B\right) $ of the eigenvalues of a nonnegative matrix $B.$ We recall
the result \cite{second} :
\begin{theorem}
\label{BAUER}If $B=\left( b_{i,j}\right) \geq 0$ of order $n\times n$ has a
positive eigenvector
\begin{equation*}
\mathbf{x}=\left[ x_{1},\ldots,x_{n}\right] ^{T}
\end{equation*}%
corresponding to $\rho \left( B\right) $ then%
\begin{equation*}
\xi \left( B\right) \leq \frac{1}{2}\max_{1\leq i<j\leq
n}\sum_{k=1}^{n}x_{k}\left\vert \frac{b_{i,k}}{x_{i}}-\frac{b_{j,k}}{x_{j}}%
\right\vert .
\end{equation*}
\end{theorem}

\begin{corollary}
\label{cor}If $B=\left( b_{i,j}\right) \geq 0$ of order $n\times n$ has a
positive eigenvector
\begin{equation*}
\mathbf{x}=\left[ x_{1},x_{2},\ldots,x_{n}\right] ^{T}
\end{equation*}%
corresponding to $\rho \left( B\right) $ then%
\begin{equation}
\frac{1}{2}\max_{1\leq i,j\leq n}\sum_{k=1}^{n}x_{k}\left\vert \frac{b_{i,k}%
}{x_{i}}-\frac{b_{j,k}}{x_{j}}\right\vert \leq \rho \left( B\right) .
\label{x8}
\end{equation}
\end{corollary}

\begin{proof} \ \ We have $B\mathbf{x}=\rho \left( B\right) \mathbf{x.}$ Then%
\begin{equation*}
\sum_{k=1}^{n}b_{i,k}\frac{x_{k}}{x_{i}}=\rho \left( B\right) \text{ and }%
\sum_{k=1}^{n}b_{j,k}\frac{x_{k}}{x_{j}}=\rho \left( B\right)
\end{equation*}%
for all $i,j.$ Hence%
\begin{equation*}
\sum_{k=1}^{n}x_{k}\left\vert \frac{b_{i,k}}{x_{i}}-\frac{b_{j,k}}{x_{j}}%
\right\vert \leq \sum_{k=1}^{n}b_{i,k}\frac{x_{k}}{x_{i}}%
+\sum_{k=1}^{n}b_{j,k}\frac{x_{k}}{x_{j}}=2\rho \left( B\right) .
\end{equation*}%
From this inequality, $\left( \ref{x8}\right)$ is immediate.
\end{proof}

Our next upper bound for $\partial _{1}^{L}(G)$ is given in the following theorem.

\begin{theorem}
If $G$ is a connected graph of order $n$ then%
\begin{equation}\label{n2}
\partial_1^{L}(G)\leq \frac{1}{2}\max_{1\leq i<j\leq n}\left\{
\begin{array}{c}
\emph{Tr}(v_{i})+\emph{Tr}(v_{j})+2d_{i,j} +\sum_{k \neq i,k \neq j} |d_{i,k} - d_{j,k}|
\end{array}%
\right\}
\end{equation}%
and this upper bound does not exceed the upper bound given in (\ref{n1}).
\end{theorem}
\begin{proof} \ \ From the proof of Theorem \ref{cota1}, we have that $\textbf{1}$ is an eigenvector corresponding to $\rho(\mathcal{B}(G))$
with $\mathcal{B}(G) =\mathcal{L}(G) +\mathbf{1}\mathbf{p}^{T}$ as defined in $\left( \ref{B}%
\right) \mathbf{.}$ Applying Theorem \ref{BAUER} to $\mathcal{B}(G)$, we obtain
\begin{equation}\label{pn2}
\partial_1^{L}(G)=\xi (\mathcal{B}(G)) \leq \frac{1}{2}%
\max_{1\leq i<j\leq n}\sum_{k=1}^{n} \vert b_{i,k}-b_{j,k} \vert .
\end{equation}
We have
\begin{equation*}
  \sum_{k=1}^{n} \vert b_{i,k}-b_{j,k} \vert = \vert b_{i,i}-b_{j,i} \vert + \vert b_{i,j}-b_{j,j} \vert +
  \sum_{k \neq i, k \neq j} \vert b_{i,k}-b_{j,k} \vert =
\end{equation*}
\begin{equation*}
  \vert \emph{Tr}(v_i) + p_i - (-d_{j,i} +p_i) \vert + \vert -d_{i,j}+p_j - \emph{Tr}(v_j) -p_j \vert +
  \end{equation*}
  \begin{equation*}
    \sum_{k \neq i, k \neq j} \vert -d_{i,k} +p_k -(-d_{j,k} + p_k) \vert =
  \end{equation*}
  \begin{equation*}
  \emph{Tr}(v_i) + \emph{Tr}(v_j) + 2d_{i,j} + \sum_{k \neq i, k \neq j} \vert d_{j,k}-d_{i,k} \vert.
  \end{equation*}
Replacing in (\ref{pn2}), (\ref{n2}) is obtained. From Corollary \ref{cor}, the right hand side of (\ref{n2})
does not exceed $\rho (\mathcal{B}(G))  =\sum_{i=1}^{n}\max_{j\neq i}d(v_i,v_j)$.
\end{proof}

We now recall the following lemma that will play an important role in getting another upper bound on $\partial_1^{L}(G).$

\begin{lemma}
\label{OR2004}
\cite{OR} If $x_1 \geq x_2 \geq \ldots \geq x_m$ are real numbers such that $\sum_{i=1}^{m}x_i =0$ then
\begin{equation*}
  x_1 \leq \sqrt{\frac{m-1}{m}\sum_{i=1}^{m}x_i^{2}}.
\end{equation*}
The equality holds if and only if $x_2=\ldots=x_m=-\frac{x_1}{m-1}$.
\end{lemma}

\begin{theorem}\label{Tr}
If $G$ is a connected graph of order $n$ then
\begin{equation}\label{n3}
  \partial_1^{L}(G) \leq \frac{2W(G)}{n-1} +\sqrt{\frac{n-2}{n-1}\bigg(\| \mathcal{L}(G)\|_{F}^{2}-\frac{(2W(G))^{2}}{n-1}\bigg)}.
\end{equation}
The equality holds if and only if $G=K_n$ or $G$ is a connected graph with three distinct distance Laplacian eigenvalues:
$\partial_1^{L}(G)$, \ $\frac{2W(G)-\partial_1^{L}(G)}{n-2}$ \ and  $0$.
\end{theorem}

\begin{proof} \ \ The distance Laplacian eigenvalues of $G$ are
\begin{equation*}
  \partial_1^{L}(G) \geq \partial_2^{L}(G) \geq \ldots \geq \partial_{n-1}^{L}(G) > \partial_n^{L}(G)=0.
\end{equation*}
Then $\sum_{i=1}^{n-1}\partial_i^{L}(G) = 2W(G)$ and $\sum_{i=1}^{n-1}(\partial_i^{L}(G))^{2} = \|\mathcal{L}(G)\|_{F}^{2}.$ Moreover
\begin{equation*}
  \sum_{i=1}^{n-1}\bigg(\partial_i^{L}(G)-\frac{2W(G)}{n-1}\bigg) = 0.
\end{equation*}
Applying Lemma \ref{OR2004}, we get
\begin{equation}\label{R}
  \partial_1^{L}(G) - \frac{2W(G)}{n-1} \leq \sqrt {\frac{n-2}{n-1}\sum_{i=1}^{n-1}\bigg(\partial_{i}^{L}(G)-\frac{2W(G)}{n-1}\bigg)^{2}}.
\end{equation}
Since
\begin{equation*}
  \sum_{i=1}^{n-1}\bigg(\partial_{i}^{L}(G)-\frac{2W(G)}{n-1}\bigg)^{2}=
  \sum_{i=1}^{n-1}\big(\partial_{i}^{L}(G)\big)^{2}-2\frac{2W(G)}{n-1}\sum_{i=1}^{n-1}\partial_{i}^{L}(G)+(n-1)\big(\frac{2W(G)}{n-1}\big)^{2}
\end{equation*}
\begin{equation*}
  =\|\mathcal{L}(G)\|_{F}^{2}-2\frac{\big(2W(G)\big)^{2}}{n-1}+\frac{\big(2W(G)\big)^{2}}{n-1}=\| \mathcal{L}(G)\|_{F}^{2}-\frac{(2W(G))^{2}}{n-1},
\end{equation*}
 the upper bounds (\ref{n3}) and (\ref{R}) are equivalent. Moreover, from Lemma \ref{OR2004}, the equality in \eqref{R} holds if and only if $$\partial_{2}^{L}(G)-\frac{2W(G)}{n-1} \ = \ \cdots \ = \ \partial_{n-1}^{L}(G)-\frac{2W(G)}{n-1} \ = \ -\frac{\partial_{1}^{L}(G)-\frac{2W(G)}{n-1}}{n-2}.$$
Therefore, the equality in \eqref{n3} holds if and only if  $\partial_{2}^{L}(G) \ = \ \cdots \ = \ \partial_{n-1}^{L}(G) \ = \ \frac{2W(G)-\partial_1^{L}(G)}{n-2}$. If $\partial_1^{L}(G) = \frac{2W(G)-\partial_1^{L}(G)}{n-2}$ then $G$ is a graph in which the multiplicity of $\partial_1^{L}(G)$ is $n-1$ and thus, from Theorem \ref{Han}, $G=K_n$. If $\partial_1^{L}(G) \neq \frac{2W(G)-\partial_1^{L}(G)}{n-2}$ then $G$ is a connected graph with three distinct distance Laplacian eigenvalues:
$\partial_1^{L}(G)$, \ $\frac{2W(G)-\partial_1^{L}(G)}{n-2}$ \ and  $0$. The proof is complete.

\end{proof}

In the following example, we apply the above upper bounds to the transmission regular graph but not degree regular graph of the smallest order \cite{Aou}.
\begin{example}\label{Ex2}
Let $G$:
\begin{center}
\begin{tikzpicture}
 \tikzstyle{every node}=[draw,circle,fill=black,minimum size=4pt,
                            inner sep=0pt]
    \draw (-1.5,-2) node (1) [label=above:$\empty$] {}
          (1.5,-2) node (2) [label=below:$\empty$] {}
         (3,0) node (3) [label=below:$\empty$] {}
          (1.5,2) node (4) [label=below:$\empty$] {}
         (-1.5,2) node (5) [label=below:$\empty$] {}
         (-3,0) node (6) [label=below:$\empty$] {}
         (-1,-1) node (7) [label=below:$\empty$] {}
          (1.5,0) node (8) [label=below:$\empty$] {}
         (-1,1) node (9) [label=below:$\empty$] {};
        \draw (1)--(2);  \draw (1)--(7); \draw (1)--(6);  \draw (2)--(3);        \draw (2)--(7); \draw (2)--(8);         \draw (3)--(8); \draw (3)--(4);  \draw (4)--(9); \draw (4)--(5);  \draw (5)--(9);        \draw (5)--(6); \draw (5)--(9);  \draw (6)--(7); \draw (9)--(6);\draw (4)--(8);
\end{tikzpicture}
\end{center}
To four decimal places, $\partial_1^{L}(G)=19.3723$ and the above upper bounds for $\partial_1^{L}(G)$ are
\begin{center}
$\begin{array}{cccccc}
 (\ref{i1}) & (\ref{d1}) & (\ref{d2}) & (\ref{n1}) & (\ref{n2}) & (\ref{n3}) \\
    29.4919 & 63 & 21.8740 & 27 & 21 & 21.4782 \\
\end{array}$
\end{center}
\end{example}

In the following proposition, for a transmission regular graph $G$, we restate the upper bounds for $\partial_1^{L}(G)$ given in (\ref{d2}) and (\ref{n3}) and we prove that the upper bound in (\ref{n3}) improves the upper bound in (\ref{d2}).

\begin{proposition}
Let $G$ be a connected $k-$ transmission regular graph of order $n$. Then
\begin{enumerate}
\item
\begin{equation}\label{r1}
  \partial_1^{L}(G) < k + \sqrt{\|\mathcal{D}(G)\|_{F}^{2}-k^{2}}
\end{equation}
\item \begin{equation}\label{r2}
  \partial_1^{L}(G) \leq \frac{n}{n-1} k + \sqrt{ \frac{n-2}{n-1} \bigg(\|\mathcal{D}(G)\|_{F}^{2}-\frac{n}{n-1} k^{2}\bigg)}.
\end{equation}
\item Let
\begin{equation*}
  c_1 = k + \sqrt{\|\mathcal{D}(G)\|_{F}^{2}-k^{2}}
\end{equation*}
and
\begin{equation*}
  c_2 = \frac{n}{n-1} k + \sqrt{ \frac{n-2}{n-1} \bigg(\|\mathcal{D}(G)\|_{F}^{2}-\frac{n}{n-1} k^{2}\bigg)}.
\end{equation*}
Hence $c_2 \leq c_1$.
\end{enumerate}
\end{proposition}
\begin{proof}
\begin{enumerate}
\item For $k-$ transmission regular graph of order $n$,
we have
\begin{equation*}
  \max_{1\leq i \leq n}\emph{Tr}(v_i) = k
\end{equation*}
 and
 \begin{equation*}
   \frac{\sum_{i=1}^{n}\big(\emph{Tr}(v_i)\big)^{2}}{n}=k^{2}.
 \end{equation*}
  Thus (\ref{r1}) is immediate from (\ref{d2}).
\item Moreover, for a such graph $G$, we have
\begin{equation*}
  \frac{2W(G)}{n-1}=\frac{nk}{n-1}
\end{equation*}
 and
 \begin{equation*}
\| \mathcal{L}(G)\|_{F}^{2}-\frac{(2W(G))^{2}}{n-1} = \| \mathcal{D}(G)\|_{F}^{2} + nk^{2}-\frac{n^{2}k^{2}}{n-1}=\| \mathcal{D}(G)\|_{F}^{2}-\frac{n}{n-1}k^{2}.
 \end{equation*}
 Hence (\ref{r2}) is immediate from (\ref{n3}).
 \item  For brevity, let $\Delta = \| \mathcal{D}(G)\|_{F}$. We will prove that $c_1-c_2 \geq 0$. We have
 \begin{equation*}
   c_1-c_2=-\frac{1}{n-1}k+\sqrt{\Delta^{2}-k^{2}}-\sqrt{ \frac{n-2}{n-1} \bigg(\Delta^{2}-\frac{n}{n-1} k^{2}\bigg)}
 \end{equation*}
Hence
\begin{equation*}
  c_1-c_2 \geq 0 \Leftrightarrow
\end{equation*}
\begin{equation*}
  \sqrt{\Delta^{2}-k^{2}}-\sqrt{ \frac{n-2}{n-1} \bigg(\Delta^{2}-\frac{n}{n-1} k^{2}\bigg)}\geq \frac{1}{n-1}k \Leftrightarrow
\end{equation*}
\begin{equation*}
\Delta^{2}-k^{2} +\frac{n-2}{n-1} \bigg(\Delta^{2}-\frac{n}{n-1} k^{2}\bigg) -2 \sqrt{\Delta^{2}-k^{2}}\sqrt{ \frac{n-2}{n-1} \bigg(\Delta^{2}-\frac{n}{n-1} k^{2}\bigg)}\geq \frac{1}{(n-1)^{2} }k^{2}\Leftrightarrow
\end{equation*}
\begin{equation*}
  \frac{2n-3}{n-1}\Delta^{2}-2k^{2} \geq 2 \sqrt{\Delta^{2}-k^{2}}\sqrt{ \frac{n-2}{n-1} \bigg(\Delta^{2}-\frac{n}{n-1} k^{2}\bigg)} \Leftrightarrow
\end{equation*}
\begin{equation*}
  \frac{(2n-3)^{2}}{(n-1)^{2}}\Delta^{4}+4k^{4}-4\frac{2n-3}{n-1}\Delta^{2}k^{2} \geq 4(\Delta^{2}-k^{2})(\frac{n-2}{n-1} \bigg(\Delta^{2}-\frac{n}{n-1} k^{2}\bigg))\Leftrightarrow
\end{equation*}
\begin{equation*}
  \bigg(\frac{(2n-3)^{2}}{(n-1)^{2}}-4\frac{n-2}{n-1}\bigg)\Delta^{4}+4\bigg(1-\frac{(n-2)n}{(n-1)^{2}}\bigg)k^{4}+
  4\bigg(-\frac{2n-3}{n-1}+\frac{n-2}{n-1}+\frac{(n-2)n}{(n-1)^{2}}\bigg)\Delta^{2}k^{2} \geq 0 \Leftrightarrow
\end{equation*}
\begin{equation*}
  \bigg((2n-3)^{2}-4(n-2)(n-1)\bigg)\Delta^{4}+4\bigg((n-1)^{2}-(n-2)n\bigg)k^{4}
\end{equation*}
\begin{equation*}
  +4\bigg(-(2n-3)(n-1)+(n-2)(n-1)+n(n-2)\bigg)\Delta^{2}k^{2} \geq 0 \Leftrightarrow
\end{equation*}
\begin{equation*}
  \Delta^{4}+4k^{4}-4\Delta^{2}k^{2} \geq 0 \Leftrightarrow
\end{equation*}
\begin{equation*}
  \bigg(\Delta^{2}-2k^{2}\bigg)^{2} \geq 0.
\end{equation*}
Since the last inequality is clearly true, it follows that $c_1 - c_2 \geq 0$.

\end{enumerate}
\end{proof}

In the next example, we apply the above upper bounds to a graph which is not a transmission regular graph.
\begin{example}
\label{ntr}
Let $G$ be the graph

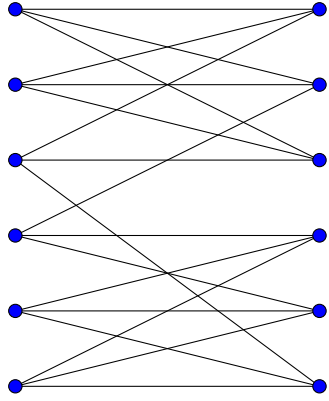
\begin{center}
\definecolor{qqqqff}{rgb}{0.,0.,1.}
\begin{tikzpicture}
\draw (0.,6.)-- (4.,6.);
\draw (0.,6.)-- (4.,5.);
\draw (0.,6.)-- (4.,4.);
\draw (0.,5.)-- (4.,6.);
\draw (0.,5.)-- (4.,5.);
\draw (0.,5.)-- (4.,4.);
\draw (0.,4.)-- (4.,6.);
\draw (0.,4.)-- (4.,4.);
\draw (0.,3.)-- (4.,5.);
\draw (0.,3.)-- (4.,3.);
\draw (0.,4.)-- (4.,1.);
\draw (0.,3.)-- (4.,2.);
\draw (0.,2.)-- (4.,3.);
\draw (0.,2.)-- (4.,2.);
\draw (0.,2.)-- (4.,1.);
\draw (0.,1.)-- (4.,3.);
\draw (0.,1.)-- (4.,2.);
\draw (0.,1.)-- (4.,1.);
\begin{scriptsize}
\draw [fill=qqqqff] (0.,1.) circle (2.5pt);
\draw [fill=qqqqff] (0.,2.) circle (2.5pt);
\draw [fill=qqqqff] (0.,3.) circle (2.5pt);
\draw [fill=qqqqff] (0.,4.) circle (2.5pt);
\draw [fill=qqqqff] (0.,5.) circle (2.5pt);
\draw [fill=qqqqff] (0.,6.) circle (2.5pt);
\draw [fill=qqqqff] (4.,1.) circle (2.5pt);
\draw [fill=qqqqff] (4.,2.) circle (2.5pt);
\draw [fill=qqqqff] (4.,3.) circle (2.5pt);
\draw [fill=qqqqff] (4.,4.) circle (2.5pt);
\draw [fill=qqqqff] (4.,5.) circle (2.5pt);
\draw [fill=qqqqff] (4.,6.) circle (2.5pt);
\draw[color=black] (2.06,5.84);
\draw[color=black] (1.98,5.36);
\draw[color=black] (1.9,4.88);
\draw[color=black] (2.12,5.36);
\draw[color=black] (2.06,4.84);
\draw[color=black] (1.98,4.36);
\draw[color=black] (2.2,4.88);
\draw[color=black] (2.06,3.84);
\draw[color=black] (2.2,3.88);
\draw[color=black] (2.06,2.84);
\draw[color=black] (1.86,2.42);
\draw[color=black] (1.98,2.36);
\draw[color=black] (2.12,2.36);
\draw[color=black] (2.06,1.84);
\draw[color=black] (1.98,1.36);
\draw[color=black] (2.2,1.88);
\draw[color=black] (2.12,1.36);
\draw[color=black] (2.06,0.84);
\end{scriptsize}
\end{tikzpicture}
\end{center}

To four decimal places, $\partial_1^{L}(G)=33.2915$ and the mentioned bounds for $\partial_1^{L}(G)$ are
\begin{center}
$\begin{array}{cccccc}
    (\ref{i1}) & (\ref{d1}) & (\ref{d2}) & (\ref{n1}) & (\ref{n2}) & (\ref{n3}) \\
    54.5307 & 164 & 38.5963 & 45 & 36 & 36.4199 \\
\end{array}$
\end{center}
\end{example}

Based on numerous numerical computations, we propose the following conjecture
\begin{conjecture}
For non transmission regular graphs the upper bound for $\partial_1^{L}(G)$ in (\ref{n3}) improves the upper bound in (\ref{d2}).
\end{conjecture}

\section{Bounds on $\partial_1^{Q}(G)$}
We begin observing that for a $k-$ transmission regular graph its distance signless Laplacian spectral radius is $2k$. In fact, $\mathcal{Q}(G)\textbf{1}=2k\textbf{1}$ and thus, by Theorem \ref{PF}, $\partial_1^{Q}(G)=2k$. For instance, the graph in Example \ref{Ex2} is $14-$ transmission regular graph and then $\partial_1^{Q}(G)=28$. \\ \\
From now on, for $i=1,\ldots,n$, $R_i(M)$ denotes the sum of the $i-$ row of a matrix $M$ of order $n$.

We recall another important result on nonnegative matrix.
\begin{theorem}
\cite{minc}
\label{fr}
Let $A = (a_{i,j})$ be an $n\times n$ nonnegative matrix with spectral radius $\rho(A)$. Then
\begin{equation}\label{row}
\min_{1 \leq i \leq n}R_i(A) \leq \rho(A) \leq \max_{1 \leq i \leq n}R_i(A).
\end{equation}
Moreover, if $A$ is an irreducible matrix, then equality holds on either side (and hence both sides) of (\ref{row}) if and only if all row sums of $A$ are equal.
\end{theorem}

\begin{corollary}
\label{cfr}
Let $G$ be a connected graph of order $n$ with $V(G)=\{v_1\ldots,v_n\}$. Then
\begin{equation}\label{tb}
  2 \min_{1 \leq i \leq n)}\emph{Tr}(v_i) \leq \partial_1^{Q}(G) \leq 2 \max_{1 \leq i \leq n)}\emph{Tr}(v_i).
\end{equation}
Moreover, equality holds on either side (and hence both sides) of (\ref{tb}) if and only if $G$ is a transmission regular graph.
\end{corollary}
\begin{proof} \ \ Since $\mathcal{Q}(G)$ is a positive matrix and $2 \emph{Tr}(v_i) = R_i(\mathcal{Q}(G))$ for $i=1,\ldots,n$, the corollary is an immediate consequence of Theorem \ref{fr}.
\end{proof}

We mention below some already known bounds on $\partial^{Q}_{1}$.

For a simple connected graph $G$ with $n$ vertices such that $Tr(v_{1})\geq Tr(v_{2})\geq \cdots \geq Tr(v_{n})$, the authors in \cite{Hong} defined the \textit{second distance degree} of a vertex $v_{i}\in V(G)$, denoted by $\mathcal{T}_{i}$, as $\mathcal{T}_{i}=\sum\limits_{k=1}^{n}d_{i,k}Tr(v_{k})$, for each $i=1,\ldots,n$.
\begin{theorem}\label{R1}
(\cite{Hong}, Theorem 3.8) Let $G$ be a simple connected graph with $n$ vertices such that $Tr(v_{1})\geq Tr(v_{2})\geq \cdots \geq Tr(v_{n})$. Then,
\begin{align}\label{i3}
\min\limits_{1\leq i \leq n}\bigg\{Tr(v_{i})+\frac{\mathcal{T}_{i}}{Tr(v_{i})}\bigg\} \ &\leq \ \partial^{Q}_{1}(G)
\end{align}
\begin{align}\label{i4}
\partial^{Q}_{1}(G) \ &\leq \ \max\limits_{1\leq i \leq n}\bigg\{Tr(v_{i})+\frac{\mathcal{T}_{i}}{Tr(v_{i})}\bigg\} \ .
\end{align}
Moreover, any equality holds if and only if $G$ has the same value \ $Tr(v_{i})+\frac{\mathcal{T}_{i}}{Tr(v_{i})}$ \ for all $i$.
\end{theorem}
\begin{theorem}\label{R2}
(\cite{Hong}, Theorem 3.9) Let $G$ be a simple connected graph with $n$ vertices such that $Tr(v_{1})\geq Tr(v_{2})\geq \cdots \geq Tr(v_{n})$. Then,
\begin{align}\label{i5}
\min\limits_{1\leq i \leq n}\bigg\{\sqrt{2\mathcal{T}_{i}+2(Tr(v_{i}))^{2}}\bigg\} \ &\leq \ \partial^{Q}_{1}(G)
\end{align}
\begin{align}\label{i6} \partial^{Q}_{1}(G) \ &\leq \ \max\limits_{1\leq i \leq n}\bigg\{\sqrt{2\mathcal{T}_{i}+2(Tr(v_{i}))^{2}}\bigg\} \ .
\end{align}
\end{theorem}
\begin{theorem}\label{R3}
(\cite{You}, Theorem 3.7) Let $G=(V,E)$ be a connected graph on $n$ vertices. Then
\begin{align}\label{i2}
\partial^{Q}_{1}(G) \ \leq \ \max\limits_{1\leq i \leq n}\bigg\{Tr(v_{i})+\sqrt{(n-1)\sum\limits_{k=1}^{n}d^{2}_{k,i}}\bigg\} \ .
\end{align}
Moreover, if the equality in (\ref{i2}) holds, then \ $Tr(v_{i})+\sqrt{(n-1)\sum\limits_{k=1}^{n}d^{2}_{k,i}} \ = \ Tr(v_{j})+\sqrt{(n-1)\sum\limits_{k=1}^{n}d^{2}_{k,j}}$ \ for any $i,j\in \{1,\ldots,n\}$.
\end{theorem}

The following lemma is proved in \cite{Li} and \cite{Liu}.

\begin{lemma} \label{LL}
Let $G$ be a connected graph and let $p(x)$ be a polynomial on $x$. Let $q_{1}(G)$ the largest signless Laplacian eigenvalue of the matrix $Q(G)=D(G)+A(G)$. Then
\begin{equation}\label{E1}
\min_{1 \leq i \leq n}\{R_{i}(p(Q(G)))\} \ \leq \ p(q_{1}(Q(G))) \ \leq \ \max_{1 \leq i \leq n}\{R_{i}(p(Q(G)))\}.
\end{equation}
Moreover, if the row sums of $p(Q(G))$ are not all equal, then both inequalities in (\ref{E1}) are strict.
\end{lemma}

The following two results are neccesary for the proof of Theorem \ref{T2} (below). In the next lemma, we extend Lemma \ref{LL} to the distance signless Laplacian matrix.
\begin{lemma}\label {L2}
Let $G$ be a connected graph and let $p(x)$ be a polynomial on $x$. Then
\begin{equation}\label{E2}
\min_{1 \leq i \leq n}\{R_{i}(p(\mathcal{Q}(G)))\}  \leq \ p(\partial^{Q}_{1}(G)) \ \leq \ \max_{1 \leq i \leq n}\{R_{i}(p(\mathcal{Q}(G)))\}.
\end{equation}
Moreover, if the row sums of $p(\mathcal{Q}(G))$ are not all equal, then both inequalities in (\ref{E2}) are strict.
\end{lemma}

\begin{proof} \ \ $\mathcal{Q}(G)$ is a positive matrix. Then there exists a positive vector $\textbf{x}=[x_1,\ldots,x_n]^{T}$ such that $\mathcal{Q}(G) \textbf{x}=\partial_1^{Q}(G)\textbf{x}$. Then
\begin{equation*}
p(\mathcal{Q}(G))\textbf{x}=p(\partial_1^{Q}(G))\textbf{x}.
\end{equation*}
We may assume $\sum_{i=1}^{n}x_i=1$. Hence
\begin{equation*}
p(\partial_1^{Q}(G)) = p(\partial_1^{Q}(G))\sum_{i=1}^{n}x_i=\sum_{i=1}^{n} p(\partial_1^{Q}(G))x_i
\end{equation*}
\begin{equation*}
=\sum_{i=1}^{n}(p(\mathcal{Q}(G))\textbf{x})_i=\sum_{i=1}^{n}x_iR_i(p(\mathcal{Q}(G))).
\end{equation*}
Since the entries of $\textbf{x}$ are positive and their is equal to $1$, we have
\begin{equation*}
\min_{1\leq i\leq n}\{R_i(p(\mathcal{Q}(G)))\} \leq \sum_{i=1}^{n}x_iR_i(p(\mathcal{Q}(G)))\leq \max_{1\leq i\leq n}\{R_i(p(\mathcal{Q}(G)))\},
\end{equation*}
and thus the result follows.
\end{proof}\ \\
\begin{theorem}
\label{T1}
Let $G$ be a connected graph on $n$ vertices. Let $\emph{T}$ and $\emph{t}$ be the maximum and the minimum transmissions of $G$, respectively. Then, for any $u\in V(G)$,
\begin{equation*}
2W(G)+(\emph{t}-1)Tr(u)-(n-1)\emph{t} \leq \sum_{v\neq u}d(u,v)Tr(v)
\end{equation*}
\begin{equation*}
   \leq \ 2W(G)+(\emph{T}-1)Tr(u)-(n-1)\emph{T}.
\end{equation*}
\end{theorem}

\begin{proof}
\begin{align}
\sum\limits_{v\neq u}d(u,v)Tr(v) \ \notag &= \ \sum\limits_{v\neq u}Tr(v)+\sum\limits_{v\neq u}(d(u,v)-1)Tr(v) \\ \notag \\ \label{E3}&= \ 2W(G)-Tr(u)+\sum\limits_{v\neq u}(d(u,v)-1)Tr(v).
\end{align}
Using (\ref{E3}), we obtain
\begin{align*}
\sum\limits_{v\neq u}d(u,v)Tr(v) \ &\geq \ 2W(G)-Tr(u)+\emph{t}\sum\limits_{v\neq u}(d(u,v)-1) \\ \\ &= \ 2W(G)-Tr(u)+\emph{t}\bigg(\sum\limits_{v\neq u}d(u,v)-(n-1)\bigg) \\ \\ &= \ 2W(G)-Tr(u)+\emph{t}(Tr(u)-(n-1)) \\ \\ &= \ 2W(G)+(\emph{t}-1)Tr(u)-(n-1)\emph{t}.
\end{align*}
Also, using (\ref{E3}), we get
\begin{align*}
\sum\limits_{v\neq u}d(u,v)Tr(v) \ &\leq \ 2W(G)-Tr(u)+\emph{T}\sum\limits_{v\neq u}(d(u,v)-1) \\ \\ &= \ 2W(G)-Tr(u)+\emph{T}\bigg(\sum\limits_{v\neq u}d(u,v)-(n-1)\bigg) \\ \\ &= \ 2W(G)-Tr(u)+\emph{T}(Tr(u)-(n-1)) \\ \\ &= \ 2W(G)+(\emph{T}-1)Tr(u)-(n-1)\emph{T}.
\end{align*}
This completes the proof.
\end{proof}

Let $\textbf{e}_i$ be the $n-$ dimensional vector of zeros except for the $i-$ entry equal to $1$.

\begin{theorem}
\label{T2} Let $G$ be a connected graph on $n$ vertices. Let $\emph{T}$ and $\emph{t}$ the maximum and the minimum transmissions of $G$, respectively. Then
\begin{align}\label{CI5}
\frac{\emph{t}-1+\sqrt{(\emph{t}-1)^{2}+8(\emph{t}^{2}+2W(G)-(n-1)\emph{t})}}{2} \ \leq \ \partial^{Q}_{1}(G) \
\end{align}
and
\begin{align}\label{CS6}
\partial^{Q}_{1}(G) \ \leq \ \frac{\emph{T}-1+\sqrt{(\emph{T}-1)^{2}+8(\emph{T}^{2}+2W(G)-(n-1)\emph{T})}}{2} \ .
\end{align}
\end{theorem}

\begin{proof} \ \ Since $\mathcal{Q}(G)=Tr(G)+\mathcal{D}(G)$, we have
\begin{equation*}
\mathcal{Q}^{2}(G) \ = \ Tr^{2}(G)+Tr(G)\mathcal{D}(G)+\mathcal{D}(G)Tr(G)+\mathcal{D}^{2}(G).
\end{equation*}
Therefore, the $i$-th row sum of $\mathcal{Q}^{2}(G)$ is
\begin{align}
R_{i}(\mathcal{Q}^{2}(G)) \ &= \textbf{e}^{T}_{i}\mathcal{Q}^{2}(G) \textbf{1} \notag \\ \notag \\ &= \ \textbf{e}^{T}_{i} Tr^{2}(G) \textbf{1}+\textbf{e}^{T}_{i}Tr(G)\mathcal{D}(G) \textbf{1}+\textbf{e}^{T}_{i} \mathcal{D}(G)Tr(G) \textbf{1}+\textbf{e}^{T}_{i} \mathcal{D}^{2}(G)\textbf{1} \notag \\ \notag \\ &= \ 2\textbf{e}^{T}_{i}Tr^{2}(G) \textbf{1}+2\textbf{e}^{T}_{i} \mathcal{D}^{2}(G)\textbf{1} \notag \\ \notag \\ \label{E4}&= \ 2Tr^{2}(v_i)+2\sum\limits_{v\neq v_i}d(v_i,v)Tr(v).
\end{align}
From Theorem \ref{T1} and (\ref{E4}), we have
\begin{equation}\label{E5}
2Tr^{2}(v_i)+2(2W(G)+(\emph{t}-1)Tr(v_i)-(n-1)\emph{t}) \ \leq \ R_{i}(\mathcal{Q}^{2}(G))
\end{equation}
and
\begin{equation}\label{E6}
R_{i}(\mathcal{Q}^{2}(G)) \ \leq \ 2Tr^{2}(v_i)+2(2W(G)+(\emph{T}-1)Tr(v_i)-(n-1)\emph{T}).
\end{equation}\ \\
Let $p(x)=x^{2}-(\emph{t}-1)x$. The $i$-th row sum of $p(\mathcal{Q}(G))$ is
\begin{align}
R_{i}(p(\mathcal{Q}(G))) \ &= \ R_{i}(\mathcal{Q}^{2}(G)-(\emph{t}-1)\mathcal{Q}(G)) \notag \\ \notag \\ &= \ R_{i}(\mathcal{Q}^{2}(G))-(\emph{t}-1)R_{i}(\mathcal{Q}(G)) \notag \\ \notag \\ &= \label{E7} \ R_{i}(\mathcal{Q}^{2}(G))-2(\emph{t}-1)Tr(v_i).
\end{align}
From (\ref{E5}) and (\ref{E7}), we obtain
\begin{align}\label{E8}
R_{i}(p(\mathcal{Q}(G))) \ &\geq \ 2Tr^{2}(v_i)+4W(G)-2(n-1)\emph{t}.
\end{align}
From (\ref{E8}), for $i=1,\ldots,n$, we have
\begin{align*}
R_{i}(p(\mathcal{Q}(G))) \ &\geq \ 2\emph{t}^{2}+4W(G)-2(n-1)\emph{t}.
\end{align*}
From this inequality and Lemma \ref{L2}, we get
\begin{align*}
2\emph{t}^{2}+4W(G)-2(n-1)\emph{t} \ \leq \ p(\partial^{Q}_{1}(G)) \ = \ (\partial^{Q}_{1}(G))^{2}-(\emph{t}-1)\partial^{Q}_{1}(G).
\end{align*}
This inequality allows to conclude the bound in (\ref{CI5}). Similarly, using the polynomial $p(x)=x^{2}-(\emph{T}-1)x$ and (\ref{E6}), the upper bound in (\ref{CS6}) can be obtained.
\end{proof}

\begin{remark} \ Since
\begin{equation*}
2\emph{t} \ \leq \ \frac{\emph{t}-1+\sqrt{(\emph{t}-1)^{2}+8(\emph{t}^{2}+2W(G)-(n-1)\emph{t})}}{2} \
\end{equation*}
and
\begin{equation*}
\frac{\emph{T}-1+\sqrt{(\emph{T}-1)^{2}+8(\emph{T}^{2}+2W(G)-(n-1)\emph{T})}}{2} \ \leq \ 2\emph{T},
\end{equation*}
Theorem \ref{T2} improves Corollary \ref{cfr}.
\end{remark}\ \\
We now recall the following result.

\begin{lemma}\cite{Aou2} \label{Aouc2}
A connected graph $G$ has only two distinct distance signless Laplacian eigenvalues if and only if $G$ is a complete graph.
\end{lemma}

Finally, we derive another new upper bound on the largest distance singless Laplacian eigenvalue.
\begin{theorem}\label{T14}
If $G$ is a connected graph of order $n$ then
\begin{equation}\label{CS7}
  \partial_1^{Q}(G) \leq \frac{2W(G)}{n} +\sqrt{\frac{n-1}{n}\bigg(\| \mathcal{Q}(G)\|_{F}^{2}-\frac{(2W(G))^{2}}{n}\bigg)}.
\end{equation}
The equality holds if and only if $G=K_n$.
\end{theorem}

\begin{proof} \ \ We have that \
$\sum_{i=1}^{n}\partial_i^{Q}(G) = 2W(G)$ and $\sum_{i=1}^{n}(\partial_i^{Q}(G))^{2} = \|\mathcal{Q}(G)\|_{F}^{2}$. \ Then
\begin{equation*}
  \sum_{i=1}^{n}\bigg(\partial_i^{Q}(G)-\frac{2W(G)}{n}\bigg) = 0.
\end{equation*}
Applying Lemma \ref{OR2004}, we get
\begin{equation}\label{RSJ}
  \partial_1^{Q}(G) - \frac{2W(G)}{n} \leq \sqrt {\frac{n-1}{n}\sum_{i=1}^{n}\bigg(\partial_{i}^{Q}(G)-\frac{2W(G)}{n}\bigg)^{2}}
\end{equation}
with equality if and only if
\begin{equation}\label{x}
  \partial_{2}^{Q}(G)-\frac{2W(G)}{n} \ = \ \cdots \ = \ \partial_{n}^{Q}(G)-\frac{2W(G)}{n} \ = \ -\frac{\partial_{1}^{Q}(G)-\frac{2W(G)}{n}}{n-1}.
\end{equation}
Since
\begin{equation*}
  \sum_{i=1}^{n}\bigg(\partial_{i}^{Q}(G)-\frac{2W(G)}{n}\bigg)^{2}=
  \sum_{i=1}^{n}\big(\partial_{i}^{Q}(G)\big)^{2}-2\frac{2W(G)}{n}\sum_{i=1}^{n}\partial_{i}^{Q}(G)+n\big(\frac{2W(G)}{n}\big)^{2}
\end{equation*}
\begin{equation*}
  =\|\mathcal{Q}(G)\|_{F}^{2}-2\frac{\big(2W(G)\big)^{2}}{n}+\frac{\big(2W(G)\big)^{2}}{n}=\| \mathcal{Q}(G)\|_{F}^{2}-\frac{(2W(G))^{2}}{n},
\end{equation*}
 the upper bound (\ref{RSJ}) is equivalent to
 \begin{equation}\label{xx}
   \partial_1^{Q}(G) \leq \frac{2W(G)}{n} +\sqrt{\frac{n-1}{n}\bigg(\| \mathcal{Q}(G)\|_{F}^{2}-\frac{(2W(G))^{2}}{n}\bigg)}
 \end{equation}
 with the necessary and sufficient condition for the equality given in (\ref{x}). We claim that the equality in (\ref{xx}) holds if and only if $G=K_n$. Suppose the equality in (\ref{xx}). Then the eigenvalues $\partial_{2}^{Q}(G), \ldots, \partial_{n}^{Q}(G)$ satisfy (\ref{x}) and thus
 \begin{equation*}
   \partial_{2}^{Q}(G)= \ldots = \partial_{n}^{Q}(G).
 \end{equation*}
 Moreover, from the fact that $\mathcal{Q}(G)$ is a positive matrix, $\partial_1^{Q}(G)$ is a simple eigenvalue. Hence the equality in (\ref{xx}) implies that $G$ has only two distinct distance signless Laplacian eigenvalues. Hence, from Lemma \ref{Aouc2}, $G=K_{n}$. Conversely, using the fact that the distance signless Laplacian eigenvalues of $K_n$ are \ $\partial_1^{Q}(K_{n})=2n-2$ \ and \ $\partial_i^{Q}(K_{n})=n-2$, \ for $i=2,\ldots,n$, one can easily see that the equality in \eqref{xx} holds. The proof is complete.
\end{proof}

Finally, we apply some bounds presented in this section to the graphs given in the following example, in each vertex is indicated its corresponding transmission. These graphs are taken from \cite{Aou1} and they are the all regular but not transmission regular graphs on $12$ vertices. As before, the results are given to four decimal places.

\begin{example} Let $G_1$ be the graph
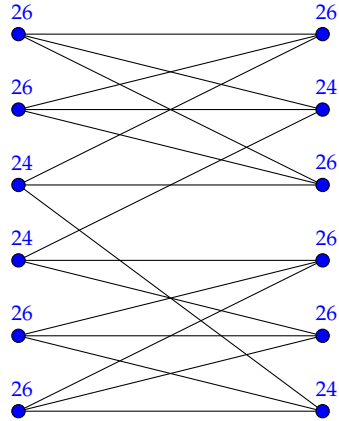
\begin{center}
\definecolor{qqqqff}{rgb}{0.,0.,1.}
\begin{tikzpicture}
\draw (0.,6.)-- (4.,6.);
\draw (0.,6.)-- (4.,5.);
\draw (0.,6.)-- (4.,4.);
\draw (0.,5.)-- (4.,6.);
\draw (0.,5.)-- (4.,5.);
\draw (0.,5.)-- (4.,4.);
\draw (0.,4.)-- (4.,6.);
\draw (0.,4.)-- (4.,4.);
\draw (0.,3.)-- (4.,5.);
\draw (0.,3.)-- (4.,3.);
\draw (0.,4.)-- (4.,1.);
\draw (0.,3.)-- (4.,2.);
\draw (0.,2.)-- (4.,3.);
\draw (0.,2.)-- (4.,2.);
\draw (0.,2.)-- (4.,1.);
\draw (0.,1.)-- (4.,3.);
\draw (0.,1.)-- (4.,2.);
\draw (0.,1.)-- (4.,1.);
\begin{scriptsize}
\draw [fill=qqqqff] (0.,1.) circle (2.5pt);
\draw[color=qqqqff] (0.05,1.3) node {$26$};
\draw [fill=qqqqff] (0.,2.) circle (2.5pt);
\draw[color=qqqqff] (0.05,2.3) node {$26$};
\draw [fill=qqqqff] (0.,3.) circle (2.5pt);
\draw[color=qqqqff] (0.05,3.3) node {$24$};
\draw [fill=qqqqff] (0.,4.) circle (2.5pt);
\draw[color=qqqqff] (0.05,4.3) node {$24$};
\draw [fill=qqqqff] (0.,5.) circle (2.5pt);
\draw[color=qqqqff] (0.05,5.3) node {$26$};
\draw [fill=qqqqff] (0.,6.) circle (2.5pt);
\draw[color=qqqqff] (0.05,6.3) node {$26$};
\draw [fill=qqqqff] (4.,1.) circle (2.5pt);
\draw[color=qqqqff] (4.05,1.3) node {$24$};
\draw [fill=qqqqff] (4.,2.) circle (2.5pt);
\draw[color=qqqqff] (4.05,2.3) node {$26$};
\draw [fill=qqqqff] (4.,3.) circle (2.5pt);
\draw[color=qqqqff] (4.05,3.3) node {$26$};
\draw [fill=qqqqff] (4.,4.) circle (2.5pt);
\draw[color=qqqqff] (4.05,4.3) node {$26$};
\draw [fill=qqqqff] (4.,5.) circle (2.5pt);
\draw[color=qqqqff] (4.05,5.3) node {$24$};
\draw [fill=qqqqff] (4.,6.) circle (2.5pt);
\draw[color=qqqqff] (4.05,6.3) node {$26$};
\draw[color=black] (2.06,5.84);
\draw[color=black] (1.98,5.36);
\draw[color=black] (1.9,4.88);
\draw[color=black] (2.12,5.36);
\draw[color=black] (2.06,4.84);
\draw[color=black] (1.98,4.36);
\draw[color=black] (2.2,4.88);
\draw[color=black] (2.06,3.84);
\draw[color=black] (2.2,3.88);
\draw[color=black] (2.06,2.84);
\draw[color=black] (1.86,2.42);
\draw[color=black] (1.98,2.36);
\draw[color=black] (2.12,2.36);
\draw[color=black] (2.06,1.84);
\draw[color=black] (1.98,1.36);
\draw[color=black] (2.2,1.88);
\draw[color=black] (2.12,1.36);
\draw[color=black] (2.06,0.84);
\end{scriptsize}
\end{tikzpicture}
\end{center}

For $G_1$, we have \ $\partial_1^{Q}(G_1)=50.8062$ \ and
\begin{center}
$
  \begin{array}{cccccccc}
    \text{Lower bounds:} & (\ref{i3}) & (\ref{i5}) & (\ref{CI5}) \\
   & 49 & 48.4974 & 48.4358 \\
  \end{array}
$
\end{center}
\begin{center}
$
  \begin{array}{cccccccc}
    \text{Upper bounds:} & (\ref{i4}) & (\ref{i6}) & (\ref{i2}) & (\ref{CS6}) & (\ref{CS7}) \\
   & 51.6923 & 51.8459 & 54.5307 & 51.7969 & 53.2578\\
  \end{array}
$
\end{center}

Consider now the graph $G_{2}$ displayed below.

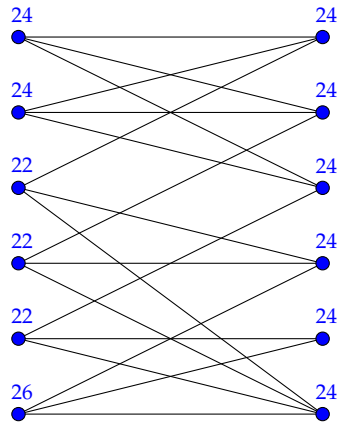
\begin{center}
\definecolor{qqqqff}{rgb}{0.,0.,1.}
\begin{tikzpicture}
\draw (0.,6.)-- (4.,6.);
\draw (0.,6.)-- (4.,5.);
\draw (0.,6.)-- (4.,4.);
\draw (0.,5.)-- (4.,6.);
\draw (0.,5.)-- (4.,5.);
\draw (0.,5.)-- (4.,4.);
\draw (0.,4.)-- (4.,6.);
\draw (0.,4.)-- (4.,3.);
\draw (0.,3.)-- (4.,5.);
\draw (0.,3.)-- (4.,3.);
\draw (0.,4.)-- (4.,1.);
\draw (0.,3.)-- (4.,1.);
\draw (0.,2.)-- (4.,4.);
\draw (0.,2.)-- (4.,2.);
\draw (0.,2.)-- (4.,1.);
\draw (0.,1.)-- (4.,3.);
\draw (0.,1.)-- (4.,2.);
\draw (0.,1.)-- (4.,1.);
\begin{scriptsize}
\draw [fill=qqqqff] (0.,1.) circle (2.5pt);
\draw[color=qqqqff] (0.05,1.3) node {$26$};
\draw [fill=qqqqff] (0.,2.) circle (2.5pt);
\draw[color=qqqqff] (0.05,2.3) node {$22$};
\draw [fill=qqqqff] (0.,3.) circle (2.5pt);
\draw[color=qqqqff] (0.05,3.3) node {$22$};
\draw [fill=qqqqff] (0.,4.) circle (2.5pt);
\draw[color=qqqqff] (0.05,4.3) node {$22$};
\draw [fill=qqqqff] (0.,5.) circle (2.5pt);
\draw[color=qqqqff] (0.05,5.3) node {$24$};
\draw [fill=qqqqff] (0.,6.) circle (2.5pt);
\draw[color=qqqqff] (0.05,6.3) node {$24$};
\draw [fill=qqqqff] (4.,1.) circle (2.5pt);
\draw[color=qqqqff] (4.05,1.3) node {$24$};
\draw [fill=qqqqff] (4.,2.) circle (2.5pt);
\draw[color=qqqqff] (4.05,2.3) node {$24$};
\draw [fill=qqqqff] (4.,3.) circle (2.5pt);
\draw[color=qqqqff] (4.05,3.3) node {$24$};
\draw [fill=qqqqff] (4.,4.) circle (2.5pt);
\draw[color=qqqqff] (4.05,4.3) node {$24$};
\draw [fill=qqqqff] (4.,5.) circle (2.5pt);
\draw[color=qqqqff] (4.05,5.3) node {$24$};
\draw [fill=qqqqff] (4.,6.) circle (2.5pt);
\draw[color=qqqqff] (4.05,6.3) node {$24$};
\draw[color=black] (2.06,5.84);
\draw[color=black] (1.98,5.36);
\draw[color=black] (1.9,4.88);
\draw[color=black] (2.12,5.36);
\draw[color=black] (2.06,4.84);
\draw[color=black] (1.98,4.36);
\draw[color=black] (2.2,4.88);
\draw[color=black] (2.06,3.84);
\draw[color=black] (2.2,3.88);
\draw[color=black] (2.06,2.84);
\draw[color=black] (1.86,2.42);
\draw[color=black] (1.98,2.36);
\draw[color=black] (2.12,2.36);
\draw[color=black] (2.06,1.84);
\draw[color=black] (1.98,1.36);
\draw[color=black] (2.2,1.88);
\draw[color=black] (2.12,1.36);
\draw[color=black] (2.06,0.84);
\end{scriptsize}
\end{tikzpicture}
\end{center}
For $G_2$, we have \ $\partial_1^{Q}(G_2)=47.5268$ \ and
\begin{center}
$
  \begin{array}{cccccccc}
    \text{Lower bounds:} & (\ref{i3}) & (\ref{i5}) & (\ref{CI5}) \\
   & 45.6364 & 44.8107 & 44.5918 \\
  \end{array}
$
\end{center}
\begin{center}
$
  \begin{array}{cccccccc}
    \text{Upper bounds:} & (\ref{i4}) & (\ref{i6}) & (\ref{i2}) & (\ref{CS6}) & (\ref{CS7}) \\
   & 50.1538 & 51.0686 & 54.5307 & 51.2847 & 49.6175\\
  \end{array}
$
\end{center}

Finally, let $G_{3}$ be the graph
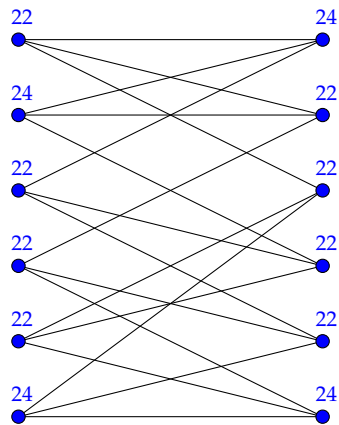
\begin{center}
\definecolor{qqqqff}{rgb}{0.,0.,1.}
\begin{tikzpicture}
\draw (0.,6.)-- (4.,6.);
\draw (0.,6.)-- (4.,5.);
\draw (0.,6.)-- (4.,4.);
\draw (0.,5.)-- (4.,6.);
\draw (0.,5.)-- (4.,5.);
\draw (0.,5.)-- (4.,3.);
\draw (0.,4.)-- (4.,6.);
\draw (0.,4.)-- (4.,3.);
\draw (0.,3.)-- (4.,5.);
\draw (0.,3.)-- (4.,2.);
\draw (0.,4.)-- (4.,2.);
\draw (0.,3.)-- (4.,1.);
\draw (0.,2.)-- (4.,4.);
\draw (0.,2.)-- (4.,3.);
\draw (0.,2.)-- (4.,1.);
\draw (0.,1.)-- (4.,4.);
\draw (0.,1.)-- (4.,2.);
\draw (0.,1.)-- (4.,1.);
\begin{scriptsize}
\draw [fill=qqqqff] (0.,1.) circle (2.5pt);
\draw[color=qqqqff] (0.05,1.3) node {$24$};
\draw [fill=qqqqff] (0.,2.) circle (2.5pt);
\draw[color=qqqqff] (0.05,2.3) node {$22$};
\draw [fill=qqqqff] (0.,3.) circle (2.5pt);
\draw[color=qqqqff] (0.05,3.3) node {$22$};
\draw [fill=qqqqff] (0.,4.) circle (2.5pt);
\draw[color=qqqqff] (0.05,4.3) node {$22$};
\draw [fill=qqqqff] (0.,5.) circle (2.5pt);
\draw[color=qqqqff] (0.05,5.3) node {$24$};
\draw [fill=qqqqff] (0.,6.) circle (2.5pt);
\draw[color=qqqqff] (0.05,6.3) node {$22$};
\draw [fill=qqqqff] (4.,1.) circle (2.5pt);
\draw[color=qqqqff] (4.05,1.3) node {$24$};
\draw [fill=qqqqff] (4.,2.) circle (2.5pt);
\draw[color=qqqqff] (4.05,2.3) node {$22$};
\draw [fill=qqqqff] (4.,3.) circle (2.5pt);
\draw[color=qqqqff] (4.05,3.3) node {$22$};
\draw [fill=qqqqff] (4.,4.) circle (2.5pt);
\draw[color=qqqqff] (4.05,4.3) node {$22$};
\draw [fill=qqqqff] (4.,5.) circle (2.5pt);
\draw[color=qqqqff] (4.05,5.3) node {$22$};
\draw [fill=qqqqff] (4.,6.) circle (2.5pt);
\draw[color=qqqqff] (4.05,6.3) node {$24$};
\draw[color=black] (2.06,5.84);
\draw[color=black] (1.98,5.36);
\draw[color=black] (1.9,4.88);
\draw[color=black] (2.12,5.36);
\draw[color=black] (2.06,4.84);
\draw[color=black] (1.98,4.36);
\draw[color=black] (2.2,4.88);
\draw[color=black] (2.06,3.84);
\draw[color=black] (2.2,3.88);
\draw[color=black] (2.06,2.84);
\draw[color=black] (1.86,2.42);
\draw[color=black] (1.98,2.36);
\draw[color=black] (2.12,2.36);
\draw[color=black] (2.06,1.84);
\draw[color=black] (1.98,1.36);
\draw[color=black] (2.2,1.88);
\draw[color=black] (2.12,1.36);
\draw[color=black] (2.06,0.84);
\end{scriptsize}
\end{tikzpicture}
\end{center}

For $G_{3}$, we have that \ $\partial_1^{Q}(G_3)=45.4891$ \ and
\begin{center}
$
  \begin{array}{cccccccc}
    \text{Lower bounds:} & (\ref{i3}) & (\ref{i5}) & (\ref{CI5}) \\
   & 44.3636 & 44.1814 & 44.2380 \\
  \end{array}
$
\end{center}
\begin{center}
$
  \begin{array}{cccccccc}
    \text{Upper bounds:} & (\ref{i4}) & (\ref{i6}) & (\ref{i2}) & (\ref{CS6}) & (\ref{CS7}) \\
   & 47 & 47.4974 & 50.1151 & 47.5590 & 47.2386 \\
  \end{array}
$
\end{center}
\end{example}

An immediate conclusion from the above example is that the lower bounds (\ref{i5}) and (\ref{CI5}) as well as the upper bounds (\ref{i6}) and (\ref{CS6}) are not comparable. In fact, (i) (\ref{CI5}) gives a better lower bound for $\partial_1^{Q}(G_3)$ than (\ref{i5}) does, but this is not the case for the graphs $G_1$ and $G_2$ in which (\ref{i5}) gives better lower bounds and (ii) (\ref{CS6}) gives a better upper bound for $\partial_1^{Q}(G_1)$ than (\ref{i6}) does, but this is not the case for the graphs $G_2$ and $G_3$ in which (\ref{i6}) gives better upper bounds. Finally, we recall if $G$ is $k-$ transmission regular graph, that is, if $t=T=k$ then $\partial_1^{Q}(G)=2k$. Then, we can expect for tight bounds in (\ref{CI5}) and (\ref{CS6}) if $0 < T-t \leq 2$.
\\ \\
\textbf{Acknowledgements. } The research of R. D\'iaz was supported by Conicyt-Fondecyt de Postdoctorado 2017 $N^{o}$ 3170065, Chile. The research of A. Julio was supported by Conicyt-PAI 79160002, 2016, Chile. The research of O. Rojo was supported by Project Fondecyt Regular 1170313, Chile.

\bigskip

E-mails:

\ \ \ Roberto D\'{i}az - \textit{rdiaz01@ucn.cl}

\ \ \ Ana Julio - \textit{ajulio@ucn.cl}

\ \ \ Oscar Rojo - \textit{orojo@ucn.cl}


\bigskip



\begin{thebibliography}{99}
%
%
%
%
%
%
%
%
%
%
%
%
%
%
%
%
%
%
%
%
%
%
%
%
%
%
%
%
%
%
%
%
%
%
%
%
%
%
%
%
%
%
%
%
%
%

\bibitem{Aou1} M. Aouchiche, P. Hansen, On a conjecture about the Szeged index, \textit{European Journal of Combinatorics}, \textbf{31} (2010) 1662-1666.

\bibitem{Aou}
M. Aouchiche, P. Hansen, Two Laplacians for the distance matrix of a graph, \textit{Linear Algebra Appl.}, \textbf{439}(2013) 21-33.

\bibitem{AuchichHansen2014}
M. Aouchiche, P. Hansen, Distance spectra of graphs: A survey, \textit{Linear Algebra Appl.}, \textbf{458} (2014): 301--386.

\bibitem{Hansen} M. Aouchiche, P. Hansen, Some properties of the distance Laplacian eigenvalues of a graph,
 \textit{Czechoslovak Mathematical Journal}, \textbf{64} (139) (2014), 751-761.

\bibitem{Aou2} M. Aouchiche, P. Hansen, On the distance signless Laplacian of a graph, \textit{Linear and Multilinear Algebra} \textbf{64} (6), (2016) 1113-1123.

\bibitem{boesh68-69}
F.T. Boesch, Properties of the distance matrix of a tree, \textit{Quart. Appl. Math.}, \textbf{26} (1968-1969): 607-609.

\bibitem{Buneman74}
P. Buneman, A note on matric properties of trees, \textit{J. Comb. Theory Ser. B} \textbf{17} (1974): 48-50.

\bibitem{second} F.L. Bauer, E. Deutsch, J. Stoer, Abschatzungen fur eigenwerte positiver linearer operatoren, \textit{Linear Algebra Appl.}, \textbf{2} (1969) 275-301.


\bibitem{brauer} ] A. Brauer, Limits for the characteristic roots of a matrix IV: Applications to stochastic matrices, \textit{Duke Mathematics Journal}, \textbf{19} (1952) 75-91.






\bibitem{hakimi_yau_64}
S.L. Hakimi, S.S. Yau, Distance matrix of a graph and its realizability, \textit{Quart. Appl. Math.}, \textbf{22} (1964): 305-317.

\bibitem{Hong} W. Hong, L. You, Some sharp bounds on the distance signless Laplacian spectral radius of graphs, \textit{arXiv:1308.3427v1}, 15 Aug 2013.



\bibitem{graham_pollak71}
R.L. Graham, H.O. Pollak, On the addressing problem for loop switching, \textit{Bell Syst. Tech. J.}, \textbf{50} (1971): 2495-2519.



\bibitem{minc} H. Minc, Nonnegative Matrices, \textit{John Wiley and Sons}, New York, (1988).

\bibitem{lin} H. Lin, Y. Hong, J. Wang, J. Shu, On the distance spectrum of graphs, \textit{Linear Algebra Appl}., \textbf{439} (2013): 1662-1669.

H. Lin, B. Wu, Y. Chen, J. Shu, On the distance and distance Laplacian eigenvalues of graphs, \textit{Linear Algera Appl}., \textbf{492} (2016): 128-135.

H. Lin, K. Ch. Das, Xharacterization of extremal graphs from distance signless Laplacian eigenvalues, \textit{Linear Algebra Appl}., \textbf{500} (2016): 77-87.

\bibitem{Li} J. S. Li, Y. L. Pan, Upper bounds for the Laplacian graph eigenvalues, \textit{Acta Math. Sin. (Engl. Ser.)}, \textbf{20} (5) (2004) 803-806.

\bibitem{Liu} H. Q. Liu, M. Lu, F. Tian, On the Laplacian spectral radius of a graph, \textit{Linear Algebra Appl.}, \textbf{376} (2004) 135-141.
\bibitem{OR} O. Rojo, H. Rojo, A decreasing sequence of upper bounds on
the largest Laplacian eigenvalue of a graph, \textit{Linear Algebra Appl}., \textbf{381} (2004) 97-116.

\bibitem{StevanovicIlic2010}
D. Stevanovi\'c, A. Ili\'c, Spectral properties of distance matrix of graphs, in I. Gutman, B. Furtula (Eds), Distance in Molecular Graph Theory, \textit{Math. Chem. Monogr.}, vol. \textbf{12}, University of Kragujevac, , Kragujevac, 2010, pp. 139-176.

\bibitem{schonberg1935}
I. Schonberg, Remarks to Maurice Fr\'echet's article "Sur la d\'efinition axiomatique d'une classe d'espacedstanci\'es vectoriellement applicable
sur l\'espace de Hilbert," \textit{Ann. of Math.}, \textbf{36} (1935): 724-732.

\bibitem{spereira90}
J.M.S. Sim\~{o}es-Pereira, An algorithm and its role in the sudy of optimal graph realizations of distance matrices, \textit{Discrete Math.}, \textbf{79} (1990): 299-312.

\bibitem{spereira87}
J.M.S. Sim\~{o}es-Pereira, A note on distance matrices with unicyclic graph realizations, \textit{Discrete Math.}, \textbf{65} (1987): 277-287.

\bibitem{spereira69}
J.M.S. Sim\~{o}es-Pereira, A note on a tree realizability of a distance matrix, \textit{J. Combin. Theory Ser. B}, \textbf{6} (1969): 303-310.

\bibitem{spereira66}
J.M.S. Sim\~{o}es-Pereira, Some results on the tree realization of a distance matrix in Th\'eory des Graphs, \textit{J.  Int. \'Etude, Dunod, Rome,} 1966, pp. 383-388.

\bibitem{varone98}
S.C. Varone, Trees related realizations of distance matrices, \textit{Discrete Math.}, \textbf{192} (1998): 337-346.


\bibitem{Yu} G. Yu, On the least distance eigenvalue of a graph, \textit{Linear Algebra Appl}., \textbf{439} (2013) 2428-2433.

\bibitem{young1938}
G. Young, A. Householder, Discussion of a set of points in terms of their mutul distances, \textit{Psychomatika} \textbf{3} (1938): 19-22.

\bibitem{You} L. You, Y. Shu, X-D Zhang, A Sharp upper bound for the spectral radius of a nonnegative matrix and applications, \textit{arXiv:1607.05883v1}, \textbf{20} Jul (2016).


\bibitem{Zho} Z. Liu, On Spectral Radius of the Distance Matrix, \textit{Appl. Anal. and Discrete Math.}, \textbf{4} (2010), 269-277.


\bibitem{Das} J. Xue, H. Lin, K.Ch. Das, J. Shu, More results on the distance (signless) Laplacian eigenvalues of graphs,
 \textit{arXiv:1705.07419v1}, 21 May (2017).


%
%


\end{thebibliography}
\end{document}